\documentclass[12pt]{article}

\usepackage{amsfonts,amssymb,amsmath,amsthm,tikz,stmaryrd,hyperref}
\usepackage{mathrsfs}
\usepackage{fullpage}
\usepackage{enumitem}
\usepackage{graphicx}
\usepackage{tikz}
\usepackage{hyperref}
\usepackage{centernot}

\newtheorem{theorem}{Theorem}
\newtheorem{corollary}[theorem]{Corollary}
\newtheorem{lemma}[theorem]{Lemma}
\newtheorem{definition}[theorem]{Definition}
\newtheorem{question}[theorem]{Question}
\newtheorem{problem}[theorem]{Problem}

\renewcommand{\emptyset}{\varnothing}
\newcommand{\FF}{\mathrm{FF}}
\newcommand{\QQ}{\mathcal{Q}}
\newcommand{\C}{\mathcal{C}}
\newcommand{\D}{\mathcal{D}}

\newcommand{\PP}{\mathcal{P}}
\newcommand{\CS}{\mathcal{S}}
\newcommand{\st}{\colon\,}
\newcommand{\comment}[1]{}
\DeclareMathOperator{\incomparable}{\parallel}
\newcommand{\seriescomp}{\olessthan}

\newcommand{\chain}[1]{\underline{#1}}
\newcommand{\twochain}{\chain{2}}
\newcommand{\twoPtwo}{\twochain + \twochain}
\newcommand{\aboveset}{A}
\newcommand{\belowset}{B}
\newcommand{\wall}{\hat}

\newcount\tmpL
\newcount\outL
\newcommand{\addL}[2]
{
	\outL=#1
	\tmpL=#2
	\advance\outL by\tmpL
}

\usetikzlibrary{calc}

\tikzstyle vertex=[circle,fill=black!35,inner sep=1pt]
\tikzstyle edge=[draw,line width=0.8pt]
\tikzset{blob/.style={rectangle,rounded corners,minimum height=0.75cm,minimum width=1.5cm,draw}}

\newcommand{\drawcrosses}[2]
{
	\begin{scope}[every path/.style={black!20}]
		\foreach \lowertheta in {45, 75, 105, 135}
		{{
			\foreach \uppertheta in {225, 255, ..., 315}
			{{
				\draw[edge] (#1.\lowertheta) -- (#2.\uppertheta) ;
			}}
		}}
	\end{scope}
}

\newcommand{\drawcrossestwo}[2]
{
	\begin{scope}[every path/.style={black!20}]
		\foreach \lowertheta in {45, 75, 105, 135}
		{{
				\draw[edge] (#1.\lowertheta) -- (#2) ;
		}}
	\end{scope}
}

\begin{document} 

\author{Kevin G. Milans, Michael C. Wigal}
\title{A Dichotomy Theorem for First-Fit Chain Partitions}
\date{\today}

\maketitle

\begin{abstract}
First-Fit is a greedy algorithm for partitioning the elements of a poset into chains.  Let $\FF(w,Q)$ be the maximum number of chains that First-Fit uses on a $Q$-free poset of width $w$.  A result due to Bosek, Krawczyk, and Matecki states that $\FF(w,Q)$ is finite when $Q$ has width at most $2$.  We describe a family of posets $\QQ$ and show that the following dichotomy holds: if $Q\in\QQ$, then $\FF(w,Q) \le 2^{c(\log w)^2}$ for some constant $c$ depending only on $Q$, and if $Q\not\in\QQ$, then $\FF(w,Q) \ge 2^w - 1$.
\end{abstract}

\section{Introduction}

A \emph{partially ordered set} or \emph{poset} is a pair $(P,\leq)$ where $P$ is a set and $\leq$ is an antisymmetric, reflexive, and transitive relation on $P$.  We use $P$ instead of $(P,\leq)$ when there is no ambiguity in simplifying this notation. We write $x > y$ when $x \geq y$ and $x \neq y$.  All posets in this paper are finite. 

Two points $x,y \in P$ are \emph{comparable} if $x \leq y$ or $y \leq x$. Otherwise, $x$ and $y$ are said to be \emph{incomparable}, denoted $x \incomparable y$. We say that $y$ covers $x$ if $y>x$  and there does not exist a point $z \in P$ such that $y > z > x$. A \emph{chain} $C$ is a set of pairwise comparable elements, and the \emph{height} of $P$ is the size of a maximum chain. An \emph{antichain} $A$ is a set of pairwise incomparable elements, and the \emph{width} of $P$ is the size of a maximum antichain.

A \emph{chain partition} of a poset $P$ is a partition of the elements of $P$ into nonempty chains.  Dilworth's theorem states that for each poset $P$, the minimum size of a chain partition equals the width of $P$.  A \emph{Dilworth partition} of $P$ is a chain partition of $P$ of minimum size.  A poset $Q$ is a \emph{subposet} of $P$ if $Q$ can be obtained from $P$ by deleting elements.  We say that $P$ is \emph{$Q$-free} if $Q$ is not a subposet of $P$.

First-Fit is a simple algorithm that constructs an ordered chain partition of a poset $P$ by processing the elements of $P$ in a given \emph{presentation order}.  Suppose that First-Fit has already partitioned $\{x_1,\ldots,x_{k-1}\}$ into chains $(C_1, \ldots, C_t)$.  First-Fit then assigns $x_k$ to the first chain $C_j$ such that $C_j \cup \{x_k\}$ is a chain; if necessary, we introduce a new chain $C_{t+1}$ containing only $x_k$.  

We are concerned with the efficiency of the First-Fit algorithm.  A classical example due to Kierstead and Smith~\cite{kierstead} shows that First-Fit may use arbitrarily many chains even on posets of width $2$.  However, Bosek, Krawczyk, and Matecki~\cite{BKM} proved that for each fixed poset $Q$ of width at most $2$, the number of chains used by First-Fit on a $Q$-free poset $P$ is bounded in terms of the width of $P$.  Let $\FF(w,Q)$ be the maximum, over all $Q$-free posets $P$ of width $w$ and all presentation orders of $P$, of the number of chains that First-Fit uses.  The upper bound on $\FF(w,Q)$ given by Bosek, Krawczyk, and Matecki's can be as large as a tower of $w$'s with a height that is linear in $|Q|$.
\comment{
From their paper, g(w,s,t) = min k such that every wall of size k contains an antichain of size w or an (s,t)-ladder.  An (s,t)-ladder has s groups of size t; our ladders are (2,t)-ladders.  Define f(w,t)=g(w+1,2,t) and we have g(w+1,2,t) \ge w^{4w}.  We also have g(w+1,s,t) is at least as large as a tower of w's of height s (in fact, probably taller).  So when Q is the (s,2)-ladder, the bound on FF(w,Q) has this behavior.
}

\subsection{Prior work}

Aside from the result of Bosek, Krawczyk, and Matecki~\cite{BKM}, prior work has focused on establishing bounds on $\FF(w,Q)$ when $Q$ is a particular poset of interest.  We outline the history briefly.

A well-known result in graph theory states that if $G$ has no induced path on $4$ vertices, then the greedy algorithm produces a proper vertex-coloring of $G$ using at most $\omega(G)$ colors, where $\omega(G)$ is the clique number of $G$.  Applied to the incomparibility graph of a poset $P$, this implies that $\FF(w,N) = w$, where $N$ is the $4$-element poset with points $\{a,b,c,d\}$ and relations $a<c$ and $b<c,d$.

Let $\chain{r}$ denote the chain with $r$ elements.  The disjoint union of posets $P$ and $Q$ is denoted $P+Q$, with each element in $P$ incomparable to every element in $Q$.  An \emph{interval order} is a poset whose elements are closed intervals with $[x_1,x_2] < [y_1,y_2]$ if and only if $x_2<y_1$.  Fishburn~\cite{Fishburn} proved that a poset $P$ is an interval order if and only if $P$ is $(\twoPtwo)$-free.  The problem of determining the performance of First-Fit on interval orders is still open, despite significant efforts by various different research groups over the years.  Currently, the best known bounds are $(5-o(1))w \le \FF(w, \twoPtwo) \le 8w$.  The lower bound is due to Kierstead, D. Smith, and Trotter~\cite{KST}.  The upper bound is due to Brightwell, Kierstead, and Trotter (unpublished), and independently Narayanaswamy and Babu~\cite{NB}, who improved on the breakthrough column construction method due to Pemmaraju, Raman, and Varadarajan~\cite{PRV}.

The interval orders are the $(\twoPtwo)$-free posets; we obtain a larger class of posets by forbidding the disjoint union of longer chains.  Bosek, Krawczyk, and Szczypka~\cite{poly2chains} showed that when $r \geq s$, $\FF(w,\chain{r} + \underline{s}) \leq (3r - 2)(w - 1)w + w$.  Joret and Milans~\cite{2chains} improved the bound to $\FF(w,(\chain{r} + \chain{s})) \leq 8(r -1)(s - 1)w$.  Dujmovi\'c, Joret, and Wood~\cite{2chains2} further improved the bound to  $\FF(w,(\chain{r} + \chain{r})) \leq  8(2r - 3)w$, which is best possible up to the constants.  

The \emph{ladder} of height $n$, denoted $L_n$, consists of two disjoint chains $x_1 < \cdots < x_n$ and $y_1 < \cdots < y_n$ with $x_i\le y_j$ if and only if $i\le j$ and no relations of the form $y_i\le x_j$.  Kiearstead and Smith~\cite{kierstead} showed that $\FF(w,L_2) = w^2$ and  $\FF(2,L_n)\le 2n$.  They also proved the general bound $\FF(w,L_n)\le w^{\gamma(\lg(w)+\lg(n))}$, where $\lg(x)$ denotes the base-2 logarithm; this result plays an important role in our main theorem.

\subsection{Our Results}

Our aim is to say something about the behavior of $\FF(w,Q)$ in terms of the structure of $Q$.  We obtain subexponential bounds on $\FF(w,Q)$ when $Q$ belongs to a particular family of posets $\QQ$, and we also give an exponential lower bound on $\FF(w,Q)$ when $Q\not\in \QQ$.  From the point of view of the First-Fit algorithm, efficiency is vastly improved if a single poset in $\QQ$ is forbidden.  From the point of view of an adversary, forcing First-Fit to use exponentially many chains requires all posets in $\QQ$ to appear. 

For each $x \in P$, we define the \emph{above set} of $x$, denoted $\aboveset(x)$, to be $\{y\in P \st y > x\}$; also, when $S$ is a set of points, we define $\aboveset(S)$ to be $\bigcup_{x\in S} \aboveset(x)$.  Similarly, the \emph{below set} of $x$, denoted $\belowset(x)$, is $\{y \in P\st y < x\}$ and we extend this to sets via $\belowset(S)=\bigcup_{x\in S}\belowset(x)$. We define $\aboveset[x] = \aboveset(x) \cup \{x\}$ and similarly for $\belowset[x]$. The \emph{series composition} of posets $S_1, \ldots, S_n$, denoted $S_1 \seriescomp \cdots \seriescomp S_n$, produces a poset $S$ which has disjoint copies of $S_1, \ldots, S_n$ arranged so that $x < y$ whenever $x \in S_i$, $y \in S_j$ and $i < j$.  The \emph{blocks} of $S$ are the subposets $S_1,\ldots,S_n$.

\section{Dichotomy Theorem}

A poset is \emph{ladder-like} if its elements can be partitioned into two chains $C_1$ and $C_2$ such that if $(x,y)\in C_1\times C_2$ and $x$ is comparable to $y$, then $x < y$.  Our first lemma shows that every ladder-like poset is contained in a sufficiently large ladder.


\begin{lemma}\label{lem:ladder-like}
If $P$ is a ladder-like poset of size $n$, then $P$ is a subposet of $L_n$. 
\end{lemma}

\begin{proof}
Let $P$ be a ladder-like poset of size $n$.  Clearly the $1$-element poset is a subposet of $L_1$, and so we may assume $n\ge 2$.  Let $C_1$ and $C_2$ be a chain partition of $P$ such that whenever $(x,y)\in C_1 \times C_2$ and $x$ and $y$ are comparable, we have $x < y$.  Suppose that $P$ has a maximum element $u$.  Recall that $L_n$ consists of chains $x_1 < \cdots < x_n$ and $y_1 < \cdots < y_n$ with $x_i\le y_j$ if and only if $i\le j$.  By induction, $P - u$ can be embedded into the copy of $L_{n-1}$ in $L_n$ induced by $\{x_1,\ldots,x_{n-1}\}\cup\{y_1,\ldots,y_{n-1}\}$.  Allowing $y_n$ to play the role of $u$ completes a copy of $P$ in $L_n$.  Next, suppose that $P$ has no maximum element.  Let $u=\max C_2$, let $S = \{v\in C_1\st v\incomparable u\}$, and let $s=|S|$.  Since $P$ has no maximum element, it follows that $s\ge 1$.  By induction, $P - S$ can be embedded in the copy of $L_{n-s}$ in $L_n$ induced by $\{x_1,\ldots,x_{n-s}\}\cup\{y_1,\ldots,y_{n-s}\}$.  Allowing $\{x_{n-s+1},\ldots,x_n\}$ to play the role of $S$ completes a copy of $P$ in $L_n$.
\end{proof}

The performance of First-Fit on a poset $P$ can be analyzed using a static structure.  A \emph{wall} of a poset $P$ is an ordered chain partition $(C_1, \ldots, C_t)$ such that for each element $x\in C_j$ and each $i<j$, there exists $y\in C_i$ such that $y\incomparable x$.  It is clear that every ordered chain partition produced by First-Fit is a wall, and conversely, each wall $W$ of $P$ is output by First-Fit when the elements of $P$ are presented in order according to $W$.  Hence, the worst-case performance of First-Fit on $P$ is equal to the maximum size of a wall in $P$.  A \emph{subwall} of a wall $W$ is obtained from $W$ by deleting zero or more of the chains in $W$.  Note that if $W$ is a wall of $P$, then each subwall of $W$ is a wall of the corresponding subposet of $P$.

For each positive integer $k$, we construct a poset called the \emph{reservoir} of width $k$, denoted $R_k$, and a corresponding wall $W_k$ of size $2^k-1$.  The reservoirs provide an example of a family of posets which are good at avoiding subposets and yet still have exponential First-Fit performance.  

\begin{theorem}\label{thm:reservoir_bound}
For each $k\ge 1$, the reservoir $R_k$ has width $k$ and a wall $W_k$ of size $2^k-1$. 
\end{theorem}
\begin{proof}

Let $R_1$ be the $1$-element poset, and let $W_1$ be the chain partition of $R_1$.  For $k\ge 2$, we first construct $R_k$ using $R_{k-1}$ and $W_{k-1}$.  Then, we give a presentation order for $R_k$ which forces First-Fit to use at least $2^k-1$ chains.  Let $W_{k-1} = (C_1,\ldots, C_m)$ where $m=2^{k-1}-1$, and for $0\le i\le m$, let $\wall S_i$ be the subwall $(C_1, \ldots, C_i)$ with corresponding subposet $S_i$.  (Although $S_0$ and $\wall S_0$ are empty, they are convenient for describing $R_k$.)   Let $S = S_m \seriescomp S_{m-1}\seriescomp \cdots \seriescomp S_0 \seriescomp R_{k-1}$.  The poset $R_k$ consists of a copy of $S$ and a chain $X$ where $X = \{x_{m+1} < \cdots < x_1\}$ and each $x_i$ satisfies $\aboveset(x_i) \cap  S = \emptyset$ and $\belowset(x_i) \cap S =  S_i \cup \cdots \cup  S_m$.  See Figure~\ref{fig:reservoir}.  

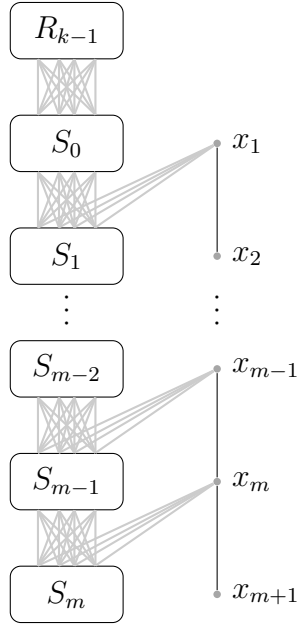
\begin{figure}
\begin{center}
\begin{tikzpicture}[yscale=0.75]
        
		\node[blob] (R) at (0,4) {$R_{k-1}$};
		\node[blob] (s0) at (0,2) {$S_0$} ;
		\node[blob] (s1) at (0,0) {$S_{1}$} ;
		\node[blob] (sm-2) at (0,-2) {$S_{m-2}$} ;
		\node[blob] (sm-1) at (0,-4) {$S_{m -1}$} ;
		\node[blob] (sm) at (0,-6) {$S_{m}$} ;
        
		\node[vertex, label=right:{$x_1$}] (x1) at (2,2) {};
		\node[vertex, label=right:{$x_{2}$}] (x2) at (2,0) {};
		\node[vertex, label=right:{$x_{m-1}$}] (xm-1) at (2,-2) {};
		\node[vertex, label=right:{$x_{m}$}] (xm) at (2,-4) {};
		\node[vertex, label=right:{$x_{m + 1}$}] (xm+1) at (2,-6) {};
		
		\path (s1) -- (sm-2) node [black, font=\small, midway, sloped] {$\dots$};
		\path (x2) -- (xm-1) node [black, font=\small, midway, sloped] {$\dots$};
		
		\draw (x1) -- (x2);
		\draw (xm-1) -- (xm) -- (xm+1);
        
        \drawcrosses{s0}{R};
        \drawcrosses{s1}{s0};
        \drawcrosses{sm-1}{sm-2};
        \drawcrosses{sm}{sm-1};
        \drawcrossestwo{s1}{x1};
        \drawcrossestwo{sm-1}{xm-1};
        \drawcrossestwo{sm}{xm};
\end{tikzpicture}
\caption{Reservoir Construction}\label{fig:reservoir}
\end{center}
\end{figure}

Note that since $S$ is a series composition of posets of width at most $k-1$, it follows that $S$ has width at most $k-1$.  Adding $X$ increases the width by at most $1$, and so $R_k$ has width at most $k$.  An antichain in the top copy of $R_{k-1}$ of size $k-1$ and $x_1$ form an antichain in $R_k$ of size $k$.

It remains to show that First-Fit might use as many as $2^k-1$ chains to partition $R_k$.  Consider the partial presentation order given by $\wall S_m, x_{m+1}, \wall S_{m-1}, x_m, \ldots, \wall S_1, x_2, \wall S_0, x_1$.  We claim that First-Fit assigns color $j$ to $x_j$ for $1 \le j\le m+1$.  Indeed, when $\hat S_{j-1}$ is presented, the points in $S_{j-1}$ are above all previously presented points except $\{x_{j+1}, \ldots, x_{m+1}\}$, which have already been assigned colors larger than $j$.  It follows that First-Fit uses colors $\{1,\ldots,j-1\}$ on $S_{j-1}$.  Next, $x_j$ is presented; since $x_j$ is above all previously presented points except those in $ S_{j-1}$, it follows that First-Fit assigns color $j$ to $x_j$.  

In the final stage, we present the top copy of $R_{k-1}$ in order given by $W_{k-1}$.  This copy of $R_{k-1}$ is incomparable to each point in $X$ and it follows that First-Fit uses $m$ new colors on these points.  In total, First-Fit uses $(m+1)+m$ colors, and $2m+1=2^k-1$.
\end{proof}

If $Q$ is a poset such that $\FF(w,Q)$ is subexponential in $w$, then Theorem~\ref{thm:reservoir_bound} implies that $Q$ is a subposet of a sufficiently large reservoir $R_k$.  These posets have a nice description.

\begin{definition}
  Let $\QQ$ be the minimal poset family which contains the ladder-like posets and is closed under series composition. 
\end{definition}

Our next lemma shows that $\QQ$ characterizes the posets of width $2$ that appear in reservoirs.

\begin{lemma}\label{lem:reservoir}
Let $Q$ be a poset of width $2$.  Some reservoir $R_k$ contains $Q$ as a subposet if and only if $Q\in\QQ$.
\end{lemma}

\begin{proof}
If $Q$ is ladder-like and has $t$ elements, then $Q$ is a subposet of $L_t$ by Lemma~\ref{lem:ladder-like}, and $L_t$ is a subposet a sufficiently large reservoir.  Suppose that $Q=Q_1\seriescomp Q_2$ for some $Q_1,Q_2\in\QQ$ with $|Q_1|,|Q_2| < |Q|$.  By induction, $Q_1$ and $Q_2$ are subposets of $R_k$ for some $k$.  Since $R_{k+1}$ contains the series composition of two copies of $R_k$, it follows that $Q$ is a subposet of $R_{k+1}$. 

Let $Q$ be a poset of width $2$ that is contained in some reservoir.  We show that $Q\in\QQ$ by induction on $|Q|$.  Let $k$ be the least positive integer such that $Q \subseteq R_k$, and let $S_0, \ldots, S_m$, $S$, and $X$ be as in the definition of $R_k$.  If $Q\cap S$ is a chain, then $(Q\cap S, Q\cap X)$ is a chain partition witnessing that $Q$ is ladder-like, and so $Q\in\QQ$.  Let $y,z$ be a maximal incomparable pair in $Q\cap S$, meaning that if $y',z'\in Q\cap S$, $y'\ge y$, $z'\ge z$ and $(y',z') \ne (y,z)$, then $y'$ and $z'$ are comparable.  We claim that if $u\in Q$ and $u$ is above one of $\{y,z\}$, then $u$ is above both $y$ and $z$.  This holds for $u\in Q\cap S$ by maximality of the pair $y,z$.  This holds for $u\in Q\cap X$ since $y\incomparable z$ implies that $y$ and $z$ belong to the same block in $S$, and all comparison relations between $u\in X$ and elements in $S$ depend only on their block in $S$.

Since $Q$ has width $2$, it follows that $Q = Q_1\seriescomp Q_2$ where $Q_1 =  \belowset [y] \cup \belowset [z]$ and $Q_2= \aboveset (y) \cup \aboveset (z)$.  Unless $Q_2$ is empty and $Q_1 = Q$, it follows by induction that $Q_1,Q_2\in\QQ$ and therefore $Q\in \QQ$ also.  Suppose that no point in $Q$ is above $y$ or $z$.  Since no point in $X$ is below a point in $S$, it follows that $Q\cap X = \emptyset$, or else a point in $Q\cap X$ would complete an antichain of size $3$ with $\{y,z\}$.  

Therefore $Q\subseteq S$.  Note that $Q$ is not contained in one of the blocks in $S$ by minimality of $k$ since each such block is a subposet of $R_{k-1}$.  It follows that $Q=Q_1\seriescomp Q_2$ for posets $Q_1$ and $Q_2$ with $|Q_1|,|Q_2|<|Q|$.  By induction, $Q_1,Q_2\in\QQ$ and so $Q\in\QQ$ also. 
\end{proof}

As a consequence of Lemma~\ref{lem:reservoir} and Theorem~\ref{thm:reservoir_bound}, it follows that $\FF(w,Q) \ge 2^w-1$ when $Q\not\in\QQ$.  It turns out that the performance of First-Fit is subexponential when $Q\in\QQ$.  Our next theorem shows how upper bounds on $\FF(w,Q_1)$ and $\FF(w,Q_2)$ can be used to obtain an upper bound on $\FF(w,Q_1\seriescomp Q_2)$.  A \emph{Dilworth coloring} of a poset $P$ of width $w$ is a function $\varphi\st P\to [w]$, where $[w]=\{1,\ldots,w\}$ such that the preimages of $\varphi$ form a Dilworth partition. 

\begin{theorem}\label{thm:series}
Let $Q_1$ and $Q_2$ be posets, let $w$, $s$, and $t$ be integers such that $\FF(w,Q_1)<s$ and $\FF(w,Q_2)<t$, and let $Q=Q_1\seriescomp Q_2$.  We have $\FF(w,Q) \le stw^2 + (s+t)w$.
\end{theorem}

\begin{proof}
For an ordered chain partition $\C$ of a poset $P$, an \emph{ascending $\C$-chain} is a chain $x_1 < \cdots < x_k$ such that the chain in $\C$ containing $x_i$ precedes the chain containing $x_j$ for $i<j$.  Similarly, a \emph{descending $\C$-chain} is a chain $x_1 > \cdots > x_k$ such that the chain in $\C$ containing $x_i$ precedes the chain containing $x_j$ for $i<j$.  The \emph{$\C$-depth} of a point $x$, denoted $d_\C(x)$, is the size of a maximum ascending $\C$-chain with bottom element $x$ and the \emph{$\C$-height} of a point $x$, denoted $h_\C(x)$, is the size of a maximum descending $\C$-chain with top element $x$.  

Let $P$ be a $Q$-free poset of width at most $w$, and let $\C$ be a wall of $P$.  We show that $|\C| \le stw^2 + (s+t)w$.  We claim that for each $x\in P$, at least one of the inequalities $h_\C(x) \le s$, $d_\C(x)\le t$ holds.  Otherwise, if $h_\C(x)\ge s+1$ and $d_\C(x)\ge t+1$, then we obtain a copy of $Q$ in $P$ as follows.  Let $x>y_1>y_2>\cdots>y_s$ be a descending $\C$-chain and let $x<z_1<z_2<\cdots < z_t$ be an ascending $\C$-chain.  Let $P_1$ be the subposet of $P$ consisting of all $u\in P$ such that for some $y_i$, the points $u$ and $y_i$ share a chain in $\C$ and $u \le y_i$.  Let $\C_1$ be the restriction of $\C$ to $P_1$ and observe that $\C_1$ is a wall of $P_1$.  Indeed, suppose that $C,C' \in \C_1$ where $C$ precedes $C'$, and let $(y_i,y_j) = (\max C, \max C')$.  Let $v\in C'$ and note that $v$ and $y_j$ share a chain in $\C$.  Let $u$ be a point in $P$ such that $u$ belongs to the same chain in $\C$ as $y_i$ and $u\incomparable v$.  Note that $u\le y_i$, since otherwise $u>y_i>y_j\ge v$, contradicting $u\incomparable v$.  Therefore $u\in P_1$ and $u\in C$.  Since $\C_1$ is a wall of $P_1$ of size $s$ and $s>\FF(w,Q_1)$, it follows that $P_1$ contains a copy of $Q_1$.  Similarly, we let $P_2$ be the subposet of $P$ consisting of all $u\in P$ such that for some $z_i$, the points $u$ and $z_i$ share a chain in $\C$ and $u \ge z_i$.  Restricting $\C$ to $P_2$ gives a wall $\C_2$ of size $t$ analogously, and since $t>\FF(w,Q_2)$, it follows that $P_2$ contains a copy of $Q_2$.  Since every element in $P_1$ is less than $x$ and $x$ is less than every element in $P_2$, it follows that $P$ contains a copy of $Q$.  

The \emph{lower part} of $P$, denoted by $L$, is $\{x\in P\st h_\C(x) \le s\}$ and the \emph{upper part} of $P$, denoted by $U$, is $P-L$.  Note that $\{L,U\}$ is a partition of $P$, that $h_\C(x) \le s$ for $x\in L$, and that $d_\C(x) \le t$ for $x\in U$.  Let $\C_U$ be the subwall of $\C$ consisting of all chains that are contained in $U$, and let $\C_{U,j}$ be the subwall of $\C_U$ consisting of the chains $C\in \C_U$ such that $d_\C(\min C) = j$.  We claim that the minimum elements of the chains in $\C_{U,j}$ form an antichain.  Suppose that $C,C'\in \C_{U,j}$ and that $C$ precedes $C'$.  Since $C$ precedes $C'$, it is not possible for $\min C > \min C'$.  Therefore if $\min C$ and $\min C'$ are comparable, then it must be that $\min C < \min C'$, and it would follow that $d_\C(\min C) > d_\C(\min C')$.  Hence $|\C_{U,j}| \le w$ for $1\le j\le t$ and so $|\C_U| \le tw$.  A symmetric argument shows that the sublist $\C_L$ consisting of all chains that are contained in $L$ satisfies $|\C_L| \le sw$.

It remains to bound the number of chains in $\C$ that contain points in both $U$ and $L$.  Let $\C_{LU}$ be the sublist of $\C$ consisting of these chains.  Note that for each $C\in \C$, we have that $y,z\in C$ and $y<z$ implies that $h_\C(y) \le h_\C(z)$ and $d_\C(y) \ge d_\C(z)$.  It follows that each point in $C\cap L$ is less than each point in $C\cap U$.  Let $\varphi\st P\to [w]$ be a Dilworth coloring.  For each $C\in \C_{LU}$ with $y=\max (C\cap L)$ and $z=\min (C \cap U)$, we assign to $C$ the \emph{signature} $(\varphi(y), h_\C(y), \varphi(z), d_\C(z))$.  We claim that the signatures are distinct.  Suppose that $C,C'\in \C_{LU}$ have the same signature and that $C$ precedes $C'$.  Let $y=\max (C\cap L)$, $z=\min (C\cap U)$, $y'=\max (C'\cap L)$, and $z' = \min (C' \cap U)$.  Note that $y<z$ is a cover relation in $C$ and $y'<z'$ is a cover relation in $C'$.  Since $\varphi(y)=\varphi(y')$, it follows that $y$ and $y'$ are comparable.  Since $h_\C(y) = h_\C(y')$, it must be that $y<y'$.  Since $\varphi(z)=\varphi(z')$, it follows that $z'$ and $z$ are comparable.  Since $d_\C(z') = d_\C(z)$, it must be that $z'<z$.  We now have that $y<z$ is a cover relation in $C$ but $y<y'<z'<z$ for points $z',y'$ that appear in a chain $C'$ that follows $C$, contradicting that $\C$ is a wall.

Since the assigned signatures are distinct, we have that $|\C_{LU}| \le stw^2$.  It follows that $|\C| \le |\C_{LU}| + |\C_L| + |\C_U| \le stw^2 + sw + tw$.
\end{proof}

\begin{corollary}\label{cor:series}
Let $Q=Q_1\seriescomp \cdots \seriescomp Q_k$.  If $\FF(w,Q_i) \le 2^{c_i(\lg w)^2}$ for $1\le i\le k$, then $\FF(w,Q) \le 2^{(c+6k)(\lg w)^2}$, where $c = \sum_{i=1}^k c_i$. 
\end{corollary}
\begin{proof}
By induction on $k$.  For $k=1$, the claim is clear.  Suppose $k\ge 2$.   Since $\FF(1,Q)\le 1$, we may assume $w\ge 2$.  Let $R=Q_1\seriescomp \cdots \seriescomp Q_{k-1}$. By induction, $\FF(w,R) \le 2^{(c'+6(k-1))(\lg w)^2}$, where $c'= \sum_{i=1}^{k-1} c_i$.  By Theorem~\ref{thm:series} with $s\le 1+2^{(c'+6(k-1))(\lg w)^2}$ and $t \le 1+2^{c_k(\lg w)^2}$, we have $\FF(w,Q) \le stw^2 + (s+t)w \le 3stw^2 < 2^2 \cdot 2^{(c'+6(k-1))(\lg w)^2 + 1} \cdot 2^{c_k(\lg w)^2 + 1} \cdot 2^{2\lg w}$.  It follows that $\lg [\FF(w,Q)] < (c' + c_k + 6(k-1))(\lg w)^2 + 4 + 2\lg w \le (c + 6k)(\lg w)^2$. 
\end{proof}

The following key result due to Kierstead and Smith~\cite{kierstead} shows that First-Fit uses a subexponential number of chains on ladder-free posets.  We follow with the characterization of posets $Q$ for which $\FF(w,Q)$ is subexponential.  

\begin{theorem}[Kierstead--Smith~\cite{kierstead}]\label{thm:ladder}
For some constant $\gamma$, we have $\FF(w,L_n) \leq w^{\gamma(\lg(w) + \lg(n))}$. 

\end{theorem}

\begin{theorem}[Dichotomy Theorem]\label{thm:dichotomy}
Let $Q$ be an $n$-element poset of width $2$.  If $Q \in \QQ$, then there exists a constant $C$ (depending only on $Q$) such that $\FF(w,Q)\le 2^{C(\lg w)^2}$; in fact, $C=O(n)$ suffices.  If $Q \notin \QQ$, then $\FF(w,Q) \geq 2^w - 1$.
\end{theorem}

\begin{proof}
Suppose $Q\not\in\QQ$.  By Theorem~\ref{thm:reservoir_bound} and Lemma~\ref{lem:reservoir}, we have $\FF(w,Q)\ge 2^w-1$.  Suppose that $Q\in\QQ$.  Since $\FF(1,Q)\le 1$, we may assume $w\ge 2$.  Since $Q\in\QQ$, it follows that $Q=Q_1\seriescomp \cdots \seriescomp Q_k$ for some ladder-like posets $Q_1,\ldots,Q_k$.  For $1\le i\le k$, let $n_i = |Q_i|$.  Since $Q_i$ is ladder-like, Theorem~\ref{thm:ladder} implies that $\FF(w,Q_i)\le 2^{c_i(\lg w)^2}$ where $c_i=\gamma(1+\frac{\lg(n_i)}{\lg(w)}) \le \gamma(1+\lg n_i)$.  By Corollary~\ref{cor:series}, it follows that $\FF(w,Q)\le 2^{(c+6k)(\lg w)^2}$, where $c=\sum_{i=1}^k c_i$. Hence, it suffices to take $C=6k + c = 6k + \sum_{i=1}^k c_i \le (6+\gamma)k + \gamma\sum_{i=1}^k \lg n_i$.  Since $\sum_{i=1}^k n_i = n$, it follows by convexity that $\sum_{i=1}^k \lg n_i \le k\lg(n/k) \le (n/e)\lg e$, where $e$ is the base of the natural logarithm.  Using $k\le n$, we conclude $C\le (6+\gamma)n + \gamma(n/e)\lg e = O(n)$.
\end{proof}

Theorem~\ref{thm:dichotomy} provides a large separation in the behavior of First-Fit on $Q$-free posets according to whether or not $Q\in\QQ$.  It may be that even stronger results are possible.  Theorem~\ref{thm:series} shows that if $\FF(w,Q_1)$ and $\FF(w,Q_2)$ are polynomial in $w$, then so is $\FF(w,Q_1\seriescomp Q_2)$.  For large $n$, the best known lower bound on $\FF(w,L_n)$ is $w^{\lg(n - 1)}/(n - 1)$, due to Bosek, Kierstead, Krawczyk, Matecki, and Smith~\cite{subbound}.  This leaves open the possibility that $\FF(w,L_n)$ is polynomial in $w$ for each fixed $n$.  If so, then the separation provided by the Dichotomy Theorem would improve, yielding that $\FF(w,Q)$ is polynomial when $Q\in\QQ$ and exponential when $Q\not\in\QQ$.  

\begin{question}\label{q:polyladder}
Is it true for each fixed $n$ that $\FF(w,L_n)$ is bounded by a polynomial in $w$?
\end{question}

It is clear that $\FF(w,L_1) = w$ and Kierstead and Smith~\cite{kierstead} proved that $\FF(w,L_2) = w^2$.  Note that $L_3=Q_1\seriescomp Q_2\seriescomp Q_3$ where $Q_1$ and $Q_3$ are $1$-element posets and $Q_2$ is the $N$ poset.  Since $\FF(w,Q_1) = \FF(w,Q_3) = 0$ and $\FF(w,Q_2) = w$, it follows from Theorem~\ref{thm:series} that $\FF(w,L_3)$ is polynomial in $w$.  A more careful analysis, along the lines of Kierstead and Smith's proof of $\FF(w,L_2)=w^2$, shows that $\FF(w,L_3)\le w^2(w+1)$.  Question~\ref{q:polyladder} is open for $n\ge 4$.

It would also be interesting to better understand the behavior of First-Fit on $Q$-free posets when $Q\not\in\QQ$.  The smallest poset of width $2$ that is not in $\QQ$ is the \emph{skewed butterfly}, denoted $\hat B$, which consists of the chains $x_1<x_2<x_3$ and $y_1 < y_2$ with relations $x_1 < y_2$ and $y_1 < x_3$.  What is $\FF(w,\hat B)$?   

\section{First-Fit on Butterfly-Free Posets}

The \emph{butterfly poset}, denoted $B$, is $Q\seriescomp Q$, where $Q$ is the $2$-element antichain.  In this section, we obtain the asymptotics of $\FF(w,B)$.  The performance of First-Fit on butterfly-free posets is strongly related to the bipartite Tur\'an number for $C_4$.  K\"ovari, S\'os, Tur\'an~\cite{4cycle} showed that the maximum number of edges in a subgraph of $K_{n,n}$ that excludes $C_4$ is $(1 + o(1))n^{3/2}$.   

\begin{lemma}[K\"ovari--S\'os--Tur\'an~\cite{4cycle}]\label{lem:projplane}
Let $q$ be a prime power, and let $n=q^2+q+1$.  There exists a $(q+1)$-regular spanning subgraph of $K_{n,n}$ that has no $4$-cycle.
\end{lemma}

We also need a standard result about the density of primes.

\begin{theorem}[Hoheisel~\cite{Hoheisel}]\label{thm:prime-density}
There exists a real number $\theta$ with $\theta < 1$ such that for all sufficiently large real numbers $x$, there is a prime in the interval $[x-x^\theta,x]$.
\end{theorem}

Since the result of Hoheisel~\cite{Hoheisel}, many research groups have improved the bound on $\theta$; see Baker and Harman~\cite{BH-history} for the history.  The current best bound is $\theta= 0.525$, due to Baker, Harman, and Pintz~\cite{BHP}.

\begin{theorem}
$\FF(w,B) \geq (1 - o(1))w^{3/2}$.
\end{theorem}

\begin{proof}
\comment{
Let $q = (-1+\sqrt{4w-3})/2$, so that $w=q^2+q+1$.  We may assume that $w$ is sufficiently large so that there is a prime $p$ in $[q - q^\theta, q]$ by Theorem~\ref{thm:prime-density}.  Let $w'=p^2+p+1$, and note that $w=q^2 + q + 1 \ge p^2 + p + 1 = w'$ and $w' = p^2+p+1 \ge (1-q^\theta/q)^2q^2 + (1-q^\theta/q)q + 1 = (1-o(1))(q^2+q+1) = (1-o(1))w$.  Hence, it suffices to show that $\FF(w',B) \ge (1-o(1))(w')^{3/2}$ since then $\FF(w,B)\ge\FF(w',B) \ge (1-o(1))(w')^{3/2} \ge (1-o(1))w^{3/2}$.}
By Theorem~\ref{thm:prime-density} and standard asymptotic arguments, we may assume that $w$ has the form $q^2+q+1$, where $q$ is prime.  By Lemma~\ref{lem:projplane}, there exists a $(q+1)$-regular $(X,Y)$-bigraph $G$ with parts of size $w$ that has no $4$-cycle.  Since $G$ is a regular bipartite graph, it follows from Hall's Theorem that $G$ has a perfect matching $M$.  Let $G'=G-M$, and let $L$ be an ordering of $E(G')$.

Using $G'$, we construct a $B$-free poset $P$ of width $w$ and a wall of $P$ size $|E(G)|$.  It will then follow that $\FF(w,B)\ge |E(G)| = (q+1)w = (1-o(1))w^{3/2}$.  Let $I_X$ be the set of all pairs $(x,e)$ such that $x\in X$, $e\in E(G')$, and $e$ is incident to $x$.  Similarly, let $I_Y$ be the set of all pairs $(y,e)$ such that $y\in Y$, $e\in E(G')$ and $e$ is incident to $y$.  We construct $P$ so that $M$ is a maximum antichain, $\belowset(M)=I_X$, and $\aboveset(M)=I_Y$.  The subposet induced by $I_X \cup M$ consists of $w$ incomparable chains, indexed by $M$.  For $x_iy_i\in M$ with $x_i\in X$ and $y_i\in Y$, the chain associated with $x_iy_i$ consists of all pairs $(x_i,e)\in I_X$ in order according to $L$ followed by top element $x_iy_i$.  The subposet induced by $M \cup I_Y$ also consists of $w$ incomparable chains, indexed by $M$.  For $x_iy_i\in M$ with $x_i\in X$ and $y_i\in Y$, the chain associated with $x_iy_i$ in the subposet induced by $M \cup I_y$ consists of bottom element $x_iy_i$ followed by all pairs $(y_i,e)\in I_Y$ in reverse order according to $L$.  Note that if $e$ is the first edge in $L$ and $e=xy$, then $(x,e)$ is minimal in $P$ and $(y,e)$ is maximal.  The chains in $I_X \cup M$ and the chains in $M \cup I_Y$ combine to form a Dilworth partition of $P$ of size $w$; let $D_i$ be the Dilworth chain containing $x_iy_i$.  It remains to describe the relations between points in $I_X$ and points in $I_Y$.  For $(x,e_1)\in I_X$ and $(y,e_2)\in I_Y$, we have that $(x,e_1)$ is covered by $(y,e_2)$ if and only if $e_1=e_2=xy\in E(G')$.

We claim that $P$ is $B$-free.  For each element $z\in I_X\cup M$, we have that $\belowset(z)$ is a chain.  Hence, a maximal element in a copy of $B$ must belong to $I_Y$.  Similarly, since $\aboveset(z)$ is a chain when $z\in M\cup I_Y$, a minimal element in a copy of $B$ must belong to $I_X$.  In a chain of cover relations from $(x,e_1)\in I_X$ up to $(y,e_2)\in I_Y$, either all points stay in the same Dilworth chain $D_i$, implying that $xy=x_iy_i\in M$, or there is a cover relation from a point in $D_i$ to a point in $D_j$, that implying $xy=x_iy_j$ with $x_iy_j\in E(G')$.  In both cases, $(x,e_1)\le (y,e_2)$ implies that $xy\in E(G)$, and it follows that a copy of $B$ in $P$ corresponds to a $4$-cycle in $G$, a contradiction.

It remains to construct a wall $W$ of $P$ of size $|E(G)|$.  The wall contains $|E(G')|$ chains of size $2$ arranged in order according to $L$, followed by $w$ singleton chains.  For $e\in L$ with $e=xy$, the corresponding chain in the wall is $(x,e)<(y,e)$.  These chains are followed by $w$ singleton chains, each consisting of a point in $M$.  Let $C_i$ and $C_j$ be chains in $W$ with $i<j$, and let $z\in C_j$.  We show that $z$ is incomparable to some point in $C_i$.  Since $M$ is an antichain, we may assume that $C_i$ is a chain of the form $(x,e)<(y,e)$.  If $C_j$ is a singleton chain containing only $z$, then $z$ is incomparable to every element in $P$ outside its Dilworth chain.  Since $(x,e)$ and $(y,e)$ are in distinct Dilworth chains, it follows that $C_i$ contains a point incomparable to $z$.  Otherwise, $C_j$ has the form $(x',e') < (y',e')$, and since $i<j$, it follows that $e$ precedes $e'$ in $L$.  Suppose that $z=(x',e')$.  If $(x',e') \incomparable (x,e)$, then $(x,e)$ is the desired point in $C_i$.  Otherwise, $(x',e')$ is comparable to $(x,e)$, implying that $(x,e)$ and $(x',e')$ are in the same Dilworth chain and $x=x'$.  Since $e$ precedes $e'$ in $L$, we have $(x,e)<(x',e')$.  If $(x',e')$ is also comparable to $(y,e)$, it must be that $(x',e')<(y,e)$.  But now $(x,e)<(x',e')<(y,e)$ contradicts that $(y,e)$ covers $(x,e)$ in $P$.  The case that $z=(y',e')$ is analogous.
\end{proof}

In a poset $P$ with a set of elements $S$, an \emph{extremal point} of $S$ is a minimal or maximal element in $S$.

\begin{lemma}\label{lem:walkdown}
Let $C$ and $D$ be chains in $P$.  If $\min C\incomparable \max D$ and $\max C\incomparable \min D$, then $C$ and $D$ are pairwise incomparable.   Consequently if $C'$ and $D'$ are chains and $(x_1,y_1),(x_2,y_2)\in C'\times D'$ are incomparable pairs, then $\min\{x_1,x_2\}\incomparable \min\{y_1,y_2\}$ and $\max\{x_1,x_2\} \incomparable \max\{y_1,y_2\}$.
\end{lemma}

\begin{proof}
If $u\le v$, $u\in C$, and $v\in D$, then $\min C\le u \le v \le \max D$.  If $u\le v$, $u\in D$, and $v\in C$, then $\min D\le u \le v\le \max C$.  For the second part, either the statement is trivial or we apply the first part to the subchains of $C'$ and $D'$ with extremal points $\{x_1,x_2\}$ and $\{y_1,y_2\}$ respectively.
\end{proof}

Starting with an arbitrary chain partition $\C$, iteratively moving elements to earlier chains produces a wall $W$ with $|W| \le |\C|$.  Beginning with a Dilworth partition, it follows that each poset $P$ of width $w$ has a \emph{Dilworth wall} consisting of $w$ chains.  If $R$ and $S$ are sets of points in $P$, we write $R<S$ if $u<v$ when $(u,v)\in R\times S$.
\begin{theorem}
$\FF(w,B) \leq (1 + o(1))w^{3/2}$. 
\end{theorem}

\begin{proof}
Let $P$ be a $B$-free poset and let $\D$ be Dilworth wall of $P$ with $\D=(D_1, \ldots, D_w)$. Let $R$ be the set of points $x\in P$ such that $\aboveset(x)$ is a chain.  Let $R' = P-R$, and note that $\belowset(x)$ is a chain for each $x\in R'$ since $P$ is $B$-free.

Let $\C$ be a wall of $P$ with $\C=(C_1, \ldots, C_t)$; we bound $|\C|$.  Since $|\D|=w$, at most $2w$ chains in $\C$ contain an extremal point from a chain in $\D$.  Also, no two chains in $\C$ are contained in the same chain in $\D$, and so at most $w$ chains in $\C$ are contained in a chain in $\D$.  Let $\C'$ be the subwall of $\C$ consisting of all chains $C\in\C$ that do not contain an extremal point of a chain in $\D$ but contain points from at least two chains in $\D$.  We have that $|\C| \le |\C'| + 3w$.  We claim that for each chain $C_i\in\C'$, we have that $C_i\cap R$ is contained in a chain in $\D$.  Suppose that $C_i\cap R$ contains elements from at least two chains in $\D$.  Let $D_{\alpha}$ be the Dilworth chain containing $\max C_i$, let $x=\max(C_i - D_\alpha)$, and let $D_{\beta}$ be the Dilworth chain containing $x$.  Let $m=\max D_\beta$, and note that $C_i\in\C'$ implies $m\not\in C_i$.  It follows that $m\in C_j$ for some $C_j\in\C$ with $j\ne i$; since $\aboveset(x)$ is a chain and $m>x$, it follows that $m$ is comparable to every element in $C_i$ and therefore $j<i$.  Let $y$ be the element covering $x$ in $C_i$.  Note that $y\in D_\alpha$ and $y$ is comparable to everything in $D_{\beta}$ since $A(x)$ is a chain, and this implies $\alpha < \beta$.  Since $m,y\in \aboveset(x)$ and $\aboveset(x)$ is a chain, either $m<y$ or $m>y$.  If $m > y$, then $m$ is comparable to everything in $D_{\alpha}$, contradicting $m\in D_\beta$ and $\alpha < \beta$.  Similarly, if $m < y$, then $y$ is comparable to every element in $C_j$, contradicting $y\in C_i$ and $j < i$.  Therefore $C_i \cap R$ is contained in a single chain in $\D$.  By a symmetric argument, $C_i\cap R'$ is contained in a single chain in $\D$.

It remains to bound $|\C'|$.  Note that for each $C\in\C'$, we have that $C\cap R$ is contained in some Dilworth chain $D_\alpha\in \D$ and $C\cap R'$ is contained on some Dilworth chain $D_\gamma\in \D$, with $\alpha \ne \gamma$; we say that $(\alpha,\gamma)$ is the \emph{signature} of $C\in \C'$ if $C\cap R\subseteq D_\alpha$ and $C\cap R'\subseteq D_\gamma$.  Note that if $C_i,C_j\in\C'$ with $i<j$, then it is not possible for both $C_i$ and $C_j$ to have the same signature $(\alpha,\gamma)$, or else $C_i\cap R' < C_j < C_i\cap R$.  Let $X$ and $Y$ be disjoint copies of $\D$, and let $G$ be the $(X,Y)$-bigraph in which $D_\alpha \in X$ and $D_\gamma\in Y$ are adjacent if and only if some chain in $\C'$ has signature $(\alpha,\gamma)$.  We claim that $G$ has no $4$-cycle, implying $|\C'| = |E(G)| \le (1+o(1))w^{3/2}$.

Suppose for a contradiction that $G$ has a $4$-cycle on $D_\alpha,D_\beta\in X$ and $D_\gamma, D_\delta\in Y$.  Let $C_i,C_j,C_k,C_\ell$ be chains in $\C'$ with signatures $(\alpha,\gamma)$, $(\alpha,\delta)$, $(\beta,\gamma)$, and $(\beta,\delta)$, respectively.  Assume, without loss of generality, that $C_i$ precedes $C_j$ in $\C$, and let $y_1\in C_j\cap R'\subseteq D_\delta$.  Since $y_1$ is in a later chain, it must be that $x_1\incomparable y_1$ for some $x_1\in C_i$.  Since $C_j\cap R$ and $C_i\cap R$ are both contained in $D_\alpha$ and $y_1\in C_j \cap R' < C_j \cap R < C_i \cap R$, it follows that $x_1\in C_i \cap R'\subseteq D_\gamma$.  Therefore there is an incomparable pair $(x_1,y_1)\in (C_i \cap R') \times (C_j \cap R')$.  A similar argument applied to $C_k$ and $C_\ell$ with top parts in $D_\beta$ shows that there is an incomparable pair $(x_2,y_2)\in (C_k\cap R')\times (C_\ell \cap R')$.  Since $C_i\cap R',C_k\cap R'\subseteq D_\gamma$ and $C_j\cap R',C_\ell\cap R'\subseteq D_\delta$, it follows from Lemma~\ref{lem:walkdown} that there is an incomparable pair $(x,y) \in D_\gamma \times D_\delta$ with $x\le \min\{\max C_i \cap R', \max C_k\cap R'\}$ and $y\le \min\{\max C_j\cap R', \max C_\ell\cap R'\}$.  Similarly, there is an incomparable pair $(x',y')\in D_\alpha \times D_\beta$ with $x'\ge \max\{\min C_i\cap R, \min C_j\cap R\}$ and $y'\ge \max\{\min C_k \cap R, \min C_\ell \cap R\}$.  Since $x,y<x',y'$, it follows that $\{x,y,x',y'\}$ induces a copy of $B$ in $P$.

Since $|\C| \le |\C'| + 3w \le (1+o(1))w^{3/2}$, the bound on $FF(w,B)$ follows.
\end{proof}

\begin{corollary}\label{cor:butterfly}
$\FF(w,B) = (1 + o(1))w^{3/2}$.
\end{corollary}

The \emph{stacked butterfly} of height $t$, denoted $B_t$, is $Q_1\seriescomp \cdots \seriescomp Q_t$, where each $Q_i$ is a $2$-element antichain. Note that $B_{2k}$ is the series composition of $k$ copies of $B$.  A consequence of our results is that $\FF(w,B_t)$ is bounded by a polynomial in $w$ for each fixed $t$.

\begin{corollary}
$\FF(w,B_{2k}) \leq (1 + o(1))w^{3.5k - 2}$
\end{corollary}

\begin{proof}
From Theorem \ref{thm:series} and Corollary \ref{cor:butterfly} we have that 
\[ \FF(w,B_{2k}) \leq (1 + o(1))w^2\FF(w,B_{2(k-1)})\FF(w,B_2) = (1 + o(1))w^{3.5k - 2}.\]
\end{proof}

It would be interesting to find lower bounds on $\FF(w,B_{2k})$.  In particular, is $\FF(w,B_{2k})$ bounded below by a polynomial in $w$ whose degree grows linearly in $k$?

\section{Conclusions and Open Problems}
A consequence of Theorem~\ref{thm:dichotomy} is that $\QQ$ is the family of posets $Q$ such that $\FF(w,Q)$ is subexponential in $w$.  It may be that $\QQ$ is also the family of posets $Q$ such that $\FF(w,Q)$ is polynomial in $w$.  This is the case if and only if Question~\ref{q:polyladder} has a positive answer.  Alternatively, if Question~\ref{q:polyladder} has a negative answer, then it would be interesting to understand what structural properties of $Q$ lead to polynomial behavior of $\FF(w,Q)$.

\begin{problem}
Characterize the posets $Q$ for which $\FF(w,Q)$ is bounded above by a polynomial in $w$.
\end{problem}

We have focused on upper bounds for posets in $\QQ$ and lower bounds for posets outside $\QQ$.  It would be nice to obtain better bounds for posets outside $\QQ$.  The smallest poset of width $2$ that is outside $\QQ$ is the \emph{skewed butterfly} $\hat B$ consisting of disjoint chains $x_1<x_2<x_3$ and $y_1<y_2$ with the cover relations $x_1 < y_2$ and $y_1<x_3$.  According to Theorem~\ref{thm:reservoir_bound}, we have $\FF(w,\hat B)\ge 2^w-1$.  What is $\FF(w,\hat B)$?  Although Bosek, Krawczyk, and Matecki~\cite{BKM} provide tower-type upper bounds on $\FF(w,Q)$, there may be room for significant improvement.  

\begin{question}\label{q:skew-butterfly}
Is there any poset $Q$ of width $2$ for which $\FF(w,Q)$ is superexponential?  
\end{question}

We have studied the behavior of First-Fit on families that forbid a single poset $Q$, but it is also natural to ask about families that forbid a set of posets.  If $\CS$ is a set of posets, we say that a poset $P$ is $\CS$-free if no poset in $\CS$ is a subposet of $P$.  Let $\FF(w,\CS)$ be the maximum number of chains that First-Fit uses on an $\CS$-free poset of width $w$.

\begin{problem}\label{p:gen-families}
Characterize the sets $\CS$ for which $\FF(w,\CS)$ is bounded by a polynomial in $w$.
\end{problem}

If $\PP$ is a poset family that is closed under taking subposets, then $\PP$ is exactly the set of posets that is $\CS$-free, where $\CS$ is the set of minimial posets not in $\PP$.  A solution to Problem~\ref{p:gen-families} is therefore equivalent to a characterization of all subposet-closed families $\PP$ such that First-Fit has polynomial behavior when restricted to $\PP$.  We suspect that this is a challenging problem, but the restriction of Problem~\ref{p:gen-families} to $|\CS|\le 2$ is likely more accessible and even partial progress would still be interesting.

\bibliographystyle{plain}
\bibliography{citations}

\end{document}